\documentclass{amsart}
\usepackage{amsmath}
\usepackage{amsfonts}
\usepackage{csquotes}
\usepackage{multicol}
\usepackage{tikz}
\usetikzlibrary{decorations.text}
\usepackage{float}

\usepackage{bm}
\usepackage{pifont}
\newcommand{\xmark}{\ding{55}}%

\usepackage{amsthm}
\usepackage{amssymb}
\usepackage{graphicx}

\makeatletter
\providecommand*{\Dashv}{%
  \mathrel{%
    \mathpalette\@Dashv\vDash
  }%
}
\newcommand*{\@Dashv}[2]{%
  \reflectbox{$\m@th#1#2$}%
}
\makeatother

\usepackage[citestyle=authoryear,bibstyle=authoryear]{biblatex}
\newtheorem{theorem}{Theorem}[section]
\newtheorem*{theorem*}{Theorem}
\newtheorem{corollary}[theorem]{Corollary}
\newtheorem{lemma}[theorem]{Lemma}
\newtheorem{definition}[theorem]{Definition}
\newtheorem{remark}[theorem]{Remark}
\newtheorem{proposition}[theorem]{Proposition}

\newtheorem{claim}{Claim}

\title{Modal Definability in Kripke's Theory of Truth}
\author{James Walsh}
\thanks{Thanks to Dan Appel, Reid Dale, Noah Schweber, Ted Slaman, and Kentar\^o Yamamoto for helpful discussion. Thanks to Neil Barton, Wes Holliday, Andrew Lee, and Matthew Mandelkern for comments on drafts.}

\begin{document}

\maketitle

\begin{abstract}
In \emph{Outline of a Theory of Truth}, Kripke introduces many of the central concepts of the logical study of truth and paradox. He informally defines some of these---such as groundedness and paradoxicality---using modal locutions. We introduce a modal language for regimenting these informal definitions. Though groundedness and paradoxicality are expressible in the modal language, we prove that intrinsicality---which Kripke emphasizes but does not define modally---is not. This follows from a characterization of the modally definable sets and relations and an attendant axiomatization of the modal semantics.
\end{abstract}

\section{Introduction}

In \emph{Outline of a Theory of Truth}, Kripke writes that
\begin{displayquote}[\cite{kripke1975outline}, p. 693]
    It has long been recognized that some of the intuitive trouble with Liar sentences is shared with such sentences as
[the truth-teller, i.e., ``This sentence is true''] which, though not paradoxical, yield no determinate truth conditions.
\end{displayquote}

So \emph{paradoxicality} is just one of the troubling features that self-referential sentences can exhibit. Others are \emph{ungroundedness} and \emph{non-intrinsicality}. One of the great successes of Kripke's theory is that it facilitates a classification of sentences according to these properties. Kripke informally describes groundedness and paradoxicality using modal locutions.
\begin{displayquote}[\cite{kripke1975outline}, p. 708, emphasis added\footnote{Except to the first instance of `could', which is emphasized in the original.}]
    \dots [The truth-teller] is ungrounded, but not paradoxical. This means that we \emph{could} consistently use the predicate `true' so as to give [the truth-teller] a truth value, though [we need] not do so.\dots [The liar] \emph{cannot} have a truth value\dots. We \emph{could} consistently use the word `true' so as to give a truth value to such a sentence as [the truth-teller] \dots. The same does not hold for the paradoxical sentences.
\end{displayquote}
Yet Kripke does not describe the intrinsic sentences using modal locutions. This is a striking contrast. Is it a mere coincidence that modal locutions appear in the description of the grounded and paradoxical sentences but not of the intrinsic sentences? Might it instead reflect some inherent discrepancy between the sets being defined?


The standard way of studying the relevant sets of sentences obviates this question. In particular, Kripke replaces the informal \emph{modal} glosses on paradoxicality and groundedness with official \emph{extensional} mathematical definitions. Nevertheless, the modal glosses are nice examples of modality occurring naturally in a logico-mathematical context, so  it is worth regimenting them in a way that retains their modal flavor. For the purposes of taking these modal definitions ``at face value,'' we introduce a modal language and an attendant interpretation $\mathbf{FIX}$. The modal language is propositional, but it contains structured atomic formulas (like in first-order logic) as opposed to propositional variables. In particular, there are atomic formulas of the form:
$$T(x) \quad \text{for ``$x$ is true.''}$$
$$F(x) \quad \text{for ``$x$ is false.''}$$ Kripke's informal modal definitions of groundedness and paradoxicality are straightforwardly formalizable in this setting.\footnote{It is is worth noting that there is other work mixing modality and Kripke's theory of truth, see \cite{stern2014modality, nicolai2021modal, cook2022outline}. However, none of these address the particular questions addressed here.}

The aforementioned interpretation $\mathbf{FIX}$ of the modal language interprets $\Box$ via unrestricted universal quantification over a set,\footnote{In particular, $\Box$ is interpreted as a universal quantifier over the set of fixed points of the Kripke jump. By `unrestricted' I just mean that the range of the quantifier is not constrained by an accessibility relation.} whence it satisfies all principles of the normal modal logic $\mathbf{S5}$. In fact, $\mathbf{FIX}$ satisfies various additional principles (note that $N(x)$ is an abbreviation for $\neg T(x) \wedge \neg F(x)$):
\begin{flalign*}
    \mathbf{Con} &\quad \neg \big(T(x) \wedge F(x)\big)\\
    \mathbf{Ground} &\quad  \big(\lozenge T(x) \wedge \lozenge F(x)\big) \to \lozenge N(x)\\
    \mathbf{Min_n} &\quad \big(\lozenge N(x_1) \wedge \dots \wedge \lozenge N(x_n)\big) \to \lozenge \big(N(x_1)\wedge\dots\wedge N(x_n)\big)
\end{flalign*}
Let $\mathbf{S5[Con, Ground,Min_n]}$ be the modal system that results from adjoining these axioms to $\mathbf{S5}$. The following fine-grained soundness and completeness theorem characterizes the modal theory of $\mathbf{FIX}$:
\begin{theorem*}
    For all formulas $\varphi$ with at most $n$ variables, the following are equivalent:
        \begin{enumerate}
            \item[(i)] $\mathbf{FIX}\Vdash \varphi$
            \item[(ii)] $\mathbf{S5[Con, Ground,Min_n]}\vdash \varphi$
        \end{enumerate}
\end{theorem*}
This completeness theorem emerges from an exact characterization of the sets (and, in general, the $n$-ary relations) that are modally definable over $\mathbf{FIX}$. From this characterization, we extract the following answer to our motivating question:
\begin{theorem*}
The set of intrinsic sentences is not modally definable.
\end{theorem*}
Thus, it is not a mere coincidence that Kripke does not lapse into modal language when defining intrinsicality.


Here is the plan for the rest of the paper. In \textsection \ref{ktt} we review the relevant features of Kripke's theory of truth. In \textsection \ref{ms} we introduce the modal language, semantics, and axiom systems that will concern us throughout the paper. In \textsection \ref{svd} we characterize the sets that are definable using modal formulas with a single variable. In \textsection \ref{mvd} we characterize the relations that are definable using modal formulas with multiple variable.



\section{Kripke's Theory of Truth}\label{ktt}

In this section we review the relevant aspects of Kripke's theory of truth \parencite{kripke1975outline}. Kripke presents a flexible framework and emphasizes that it can be developed in various specific ways. We will focus on a specific implementation that Kripke highlights, namely, the strong Kleene fixed point construction.

Let $\mathcal{L}$ be the first-order language of arithmetic. Let $\mathcal{L}_T$ be the language of arithmetic augmented with the unary predicate $\mathsf{True}(x)$. Some presentations also include a predicate $\mathsf{False}(x)$, but one may interpret falsity as truth-of-negation, whence the second predicate is strictly unnecessary. As usual, $\ulcorner\varphi\urcorner$ is the G\"odel number of the $\mathcal{L}_T$ sentence $\varphi$ according to some standard G\"odel numbering.

\subsection{Strong Kleene Logic}

Kripke defines his interpretations of $\mathcal{L}_T$ using the trivalent logic $\mathsf{K3}$, also known as the strong Kleene logic. According to $\mathsf{K3}$, a structure $w$ can stand in one of three semantic relations with a formula $\varphi$: 
\begin{itemize}
\item $w\vDash_{\mathsf{K3}}\varphi$, i.e., $\varphi$ is \textsc{true} in $w$
\item $w\Dashv_{\mathsf{K3}}\varphi$, i.e., $\varphi$ is \textsc{false} in $w$
\item $w\uparrow_{\mathsf{K3}}\varphi$ i.e., $\varphi$ is \textsc{neither} in $w$
\end{itemize}

The semantics of conjunction is given by the following clauses:
\begin{itemize}
\item $w\vDash_{\mathsf{K3}}\varphi\wedge\psi$ if $w\vDash_{\mathsf{K3}}\varphi$ and $w\vDash_{\mathsf{K3}}\varphi$
\item $w\Dashv_{\mathsf{K3}}\varphi\wedge\psi$ if $w\Dashv_{\mathsf{K3}}\varphi$ or $w\Dashv_{\mathsf{K3}}\psi$
\item $w\uparrow_{\mathsf{K3}}\varphi$ otherwise.
\end{itemize}

The semantics of negation is given by the following clauses:
\begin{itemize}
\item $w\vDash_{\mathsf{K3}}\neg\varphi$ if $w\Dashv_{\mathsf{K3}}\varphi$
\item $w\Dashv_{\mathsf{K3}}\neg\varphi$ if $w\vDash_{\mathsf{K3}}\varphi$
\item $w\uparrow_{\mathsf{K3}}\neg\varphi$ otherwise.
\end{itemize}

The semantics of the universal quantifier is given by the following clauses:
\begin{itemize}
\item $w\vDash_{\mathsf{K3}}\forall x \; \varphi(x)$ if for every $d\in w$, $w\vDash_{\mathsf{K3}}\varphi(d)$.
\item $w\Dashv_{\mathsf{K3}}\forall x \; \varphi(x)$ if for some $d\in w$, $w\Dashv_{\mathsf{K3}}\varphi(d)$.
\item $w\uparrow_{\mathsf{K3}}\forall x \; \varphi(x)$ otherwise.
\end{itemize}

\subsection{Classification of Sentences}

Kripke introduces a large collection of 3-valued (\textsc{true}, \textsc{false}, \textsc{neither}) models for $\mathcal{L}_T$. These 3-valued models are typically called ``fixed points,'' since Kripke defines them as the fixed points of a certain function (namely, the Kripke jump). The details of this function are not important for us, but the ``fixed points'' terminology is ubiquitous, so we will continue to use it. For work on the structure of these fixed points see \cite{cantini1989notes, fitting1986notes}.

\begin{remark}\label{consistency-remark}
    Each fixed point $w$ has the following properties:
    \begin{enumerate}
        \item For each $\varphi$, $\varphi$ and $\mathsf{True}(\ulcorner\varphi\urcorner)$ have the same semantic value in $w$.
        \item $w$ is consistent, i.e., there is no $\varphi$ such that $w\vDash_{\mathsf{K3}}\varphi$ and $w\vDash_{\mathsf{K3}}\neg \varphi$.\footnote{Like Kripke, we restrict our attention to consistent fixed points. It is now common to study all fixed points without this restriction.}
        \item Hence, there is no $\varphi$ such that $w\vDash_{\mathsf{K3}}\mathsf{True}(\ulcorner\varphi\urcorner) \wedge \mathsf{True}(\ulcorner\neg \varphi\urcorner)$
    \end{enumerate}
\end{remark}

For any formula $\varphi(x)$ in the language of $\mathcal{L}_T$, a standard self-reference construction yields an $\mathcal{L}_T$ sentence $\psi$ which ``says'' $\varphi(\ulcorner\psi\urcorner)$. In particular, $\psi$ and $\varphi(\ulcorner\psi\urcorner)$ have the same semantic value in each fixed point. For instance, there are $\mathcal{L}_T$ sentences $\lambda,\tau,\gamma$ such that:
\begin{itemize}
    \item $\lambda$ says $\neg \mathsf{True}(\ulcorner\lambda\urcorner)$.
    \item $\tau$ says $ \mathsf{True}(\ulcorner\tau\urcorner)$.
    \item $\gamma$ says $ \mathsf{True}(\ulcorner\gamma\urcorner)\vee\neg\mathsf{True}(\ulcorner\gamma\urcorner)$.
\end{itemize}
In fact, there are many such sentences. We will sometimes appeal to this fact, which is well-known in the theory of truth literature:
\begin{remark}\label{truth-teller}
    There are infinitely many logically independent truth-teller sentences. Hence, there is an infinite collection $T$ of sentences such that for any sequence $\tau_1,\dots,\tau_n$ of sentences from $T$ and any assignment of \textsc{true}, \textsc{false}, and \textsc{neither} to $\tau_1,\dots,\tau_n$, there is a fixed point realizing exactly that truth-value assignment.
\end{remark}

\begin{definition}
    A sentence is \emph{grounded} just in case (i) it has the value \textsc{true} in all models or (ii) it has the value \textsc{false} in all models.
\end{definition}

\begin{remark}\label{min-fix-point}
    There is a \emph{minimum} fixed point in which all ungrounded sentences have the value \textsc{neither}. Hence, a sentence is \emph{ungrounded} just in case it has the value \textsc{neither} in some fixed point. The following claims follow:
    \begin{enumerate}
        \item If a sentence has value \textsc{true} in one fixed point and \textsc{false} in another, then it has value \textsc{neither} in the minimum fixed point.
        \item If $\varphi_1,\dots,\varphi_n$ are ungrounded, then they all have value \textsc{neither} in the minimum fixed point.
    \end{enumerate}
\end{remark}

The liar $\lambda$ , truth-teller $\tau$, and tautology-teller $\gamma$ are all ungrounded, so they all have value \textsc{neither} in the minimum model.

\begin{definition}
    A sentence is \emph{paradoxical} just in case it has value \textsc{neither} in all models.
\end{definition}

The liar $\lambda$ is paradoxical. But the truth-teller $\tau$ and tautology-teller $\gamma$ are not paradoxical:
\begin{itemize}
    \item $\tau$ has value \textsc{true} in some models and value \textsc{false} in others.
    \item $\gamma$ has value \textsc{true} in some models and does not have the value \textsc{false} in any models.
\end{itemize}

Though neither $\tau$ nor $\gamma$ is paradoxical, it is worth noting that $\gamma$ alone cannot take on both of the binary truth values. We introduce the following definition to discuss this property of $\gamma$:
\begin{definition}
    A sentence is \emph{inevitable} if it has one of the binary truth-values in some fixed point and the other binary truth-value in no fixed point.
\end{definition}
Any grounded sentence is inevitable. There are also ungrounded inevitable sentences. For instance, if $\tau$ is a truth-teller, then the sentence $\tau\vee\neg\tau$ is inevitable (inevitably true) as is $\tau\wedge\neg\tau$ (inevitably false). 

\emph{Informally}, the intrinsic sentences are those inevitable sentences whose truth/falsity does not depend on any \emph{arbitrary} sentences. Here is a more precise definition:

\begin{definition}
    A fixed point $w$ is \emph{intrinsic} if for every fixed point $v$, 
    $$\{\varphi \mid \varphi \text{ is \textsc{true} in }w\}\cup \{\varphi \mid \varphi \text{ is \textsc{true} in }v\}$$ 
    is consistent. A sentence $\varphi$ is \emph{intrinsic} if it is \textsc{true} or \textsc{false} in an intrinsic fixed point.
\end{definition}

Note that inevitability and intrinsicality come apart. The tautology-teller $\gamma$ is both intrinsic and inevitable. By contrast, $\tau\vee\neg\tau$ is inevitable but not intrinsic. The issue is that any fixed point which makes $\tau\vee\neg\tau$ \textsc{true} must either make $\tau$ \textsc{true} or $\neg\tau$ \textsc{true}. So $\tau\vee\neg\tau$ can only be made \textsc{true} by making a binary truth-value assignment that is incompatible with other possible binary truth-value assignments.


\section{Modal Systems}\label{ms}

In this section we introduce our modal framework, including the language, semantics, and axiom systems. Let's start with a presentation of the modal language. In what follows:
\begin{itemize}
    \item $\mathcal{X}$ is a countably infinite set of variables.
    \item $T$ and $F$ are one-place predicate symbols.
    \item $\neg$, $\wedge$, and $\Box$ are sentential connectives.
\end{itemize}

\begin{definition}
    The formulas of $\mathcal{L}_\Box$ are defined as follows:
\[ \varphi::= T(x) \mid F(x) \mid \neg\varphi\mid (\varphi\wedge\varphi) \mid \Box \; \varphi  \]
\end{definition}
That is, a formula is a predicate applied to a variable, a negation of a formula, a conjunction of formulas, or a box appended to a formula. Though $\vee,\to,\leftrightarrow,\lozenge$ are not officially part of the language we sometimes use these symbols as abbreviations in the standard way. Note that the language $\mathcal{L}_\Box$ neither contains quantifiers nor propositional variables.

For the rest of this paper we will be working with this signature, i.e., with predicates $T$ and $F$. For instance, when we discuss the modal logic $\mathbf{S5}$, we mean the modal logic $\mathbf{S5}$ \emph{in this signature}.

\begin{definition}
    $\mathbf{S5}$ is the classical modal logic containing all instances of the following axioms:
\begin{flalign*}
    \mathbf{K} &\quad \Box(A\to B)\to (\Box A\to \Box B)\\
    \mathbf{T} &\quad \Box A \to A\\
    \mathbf{5} &\quad \lozenge A \to \Box\lozenge A
\end{flalign*}
and closed under the following rule:
\begin{flalign*}
    \mathbf{RN} &\quad \text{If $A$ is a theorem, so is $\Box A$}.
\end{flalign*}
\end{definition}

We will now introduce models for interpreting the formal system $\mathbf{S5}$.

\begin{definition}
    For a set $W$, a \emph{variable assignment} $V$ is a pair of functions $V_1,V_2:\mathcal{X}\to \mathcal{P}(W)$.
\end{definition}

\begin{definition}\label{model}
Given $w\in W$ and variable assignment $V$, we define $\Vdash_V$ as follows:
\begin{itemize}
\item $W,w\Vdash_V T(x)$ if $w\in V_1(x)$.
\item $W,w\Vdash_V F(x)$ if $w\in V_2(x)$.
\item $W,w\Vdash_V\neg\varphi$ if $W,w\nVdash_V\varphi$.
\item $W,w\Vdash_V\varphi\wedge\psi$ if $W,w\Vdash_V\varphi$ and $W,w\Vdash_V\psi$.
\item $W,w\Vdash_V\Box\;\varphi$ if for all $v\in W$, $W,v\Vdash_V\varphi$.
\end{itemize}
We say that $W\Vdash_V \varphi$ if for all $w\in W$, $W,w\Vdash_V \varphi$.
\end{definition}

We will be particularly interested in a particular $\mathbf{S5}$ structure. 
\begin{definition}
    Let $\mathbf{FIX}$ be the set of Kripke's fixed points, i.e., the set of Kripke's 3-valued $\mathcal{L}_T$ models. An \emph{$\mathcal{L}_T$-realization} is a function $\star: \mathcal{X} \to \mathcal{L}_T$, i.e., a mapping from variables to $\mathcal{L}_T$ formulas. Each $\mathcal{L}_T$-realization $\star$ determines a variable assignment $V^\star$:
\begin{enumerate}
    \item[(i)] $V^\star_1(x)$ is the set of worlds in which $x^\star$ is \textsc{true}.
    \item[(ii)] $V^\star_2(x)$ is the set of worlds in which $x^\star$ is \textsc{false}.
\end{enumerate}

 $\mathbf{FIX}\Vdash \varphi$ just in case $\mathbf{FIX}\Vdash_{V^\star} \varphi$ for every $\mathcal{L}_T$-realization $\star$.
\end{definition}

For any $\star$, $\mathbf{FIX}\Vdash_{V^\star}\mathbf{S5}$. In fact, $\mathbf{FIX}$ validates other principles that can be stated in the modal language $\mathcal{L}_\Box$.

\begin{definition}
    $\mathbf{S5[Con]}$ extends $\mathbf{S5}$ with the following axiom:
$$\mathbf{Con} \quad \neg \big(T(x)\wedge F(x)\big).$$
\end{definition}

In what follows, $N(x)$ is an abbreviation for the formula:
 $$\neg T(x) \wedge \neg F(x).$$

 \begin{definition}
$\mathbf{S5[Con,Ground]}$ extends $\mathbf{S5[Con]}$ with the following axiom:
$$\mathbf{Ground} \quad  (\lozenge T(x) \wedge \lozenge F(x)) \to \lozenge N(x) .$$ \end{definition}

We will introduce more axioms as they are needed. For now we note the following.

 \begin{proposition}\label{sound-1}
     For any $\star$, $\mathbf{FIX}\Vdash_{V^\star} \mathbf{S5[Con,Ground]}$. 
 \end{proposition}
 \begin{proof}
     By Remark \ref{consistency-remark} and Remark \ref{min-fix-point}.
 \end{proof}


\section{Single Variable Definability}\label{svd}

The motivations discussed in the introduction concerned definability of \emph{sets}, such as the set of grounded sentences. To study definability of sets, it suffices to focus on 1-formulas (i.e., formulas with a single variable). Accordingly, in this section we will limit our attention to 1-formulas. We will turn to $n$-formulas for $n>1$ in subsequent sections. Characterizing the definable sets is much easier than characterizing the definable $n$-ary relations, so this section will also serve as a warm up for the later sections.

\begin{definition}
An $\mathcal{L}_\Box$ formula $\varphi(x)$ \emph{defines} $\mathcal{A}$ over $\mathbf{FIX}$ just in case for every $\mathcal{L}_T$-realization $\star$:
$$\mathbf{FIX}\Vdash_{V^\star} \varphi(x) \Longleftrightarrow  x^\star \in \mathcal{A}.$$
\end{definition}

Let's provide some examples:
\begin{flalign*}
    x\text{ is grounded }&:= \Box\; T(x) \vee \Box \; F(x)\\
    x\text{ is paradoxical }&:= \neg \lozenge T(x) \wedge \neg \lozenge F(x)\\
    x\text{ is inevitable }&:= \big(\lozenge T(x) \wedge\neg \lozenge F(x)\big) \vee \big(\neg \lozenge T(x) \wedge \lozenge F(x)\big)
\end{flalign*}

One conspicuous absence from the list is \emph{intrinsicality}. This is not an oversight. One goal of this section is to show that the intrinsic sentences have no modal definition.

\subsection{Normal Forms}

One of our main tools will be a normal form theorem due to \cite{carnap1946modalities}. Let's begin by drawing a distinction between two kinds of formulas.

\begin{definition}
    An \emph{extensional} formula is one without any modal operators. An \emph{intensional} formula is one all of whose variables occurs only in the scope of a modal operator.
\end{definition}

Note that the distinction between extensional and intensional formulas is exclusive but not exhaustive.

\begin{definition}
    Let $\Gamma$ be a set of formulas and $T$ a modal system. A formula $\varphi$ is \emph{$\Gamma$-maximal} for $T$ if each of the following holds:
    \begin{enumerate}
        \item $T+\varphi$ is consistent
        \item $\varphi \in \Gamma$
        \item For every $\psi\in \Gamma$, one of the following holds:
    \begin{enumerate}
        \item $T\vdash \varphi \to \psi$
        \item $T \vdash \varphi\to \neg \psi$
    \end{enumerate}
    \end{enumerate}
\end{definition}

The following notion is inspired by the model-theoretic notion of an isolated type:
\begin{definition}
    The \emph{extensional 1-isolators} for $T$ are the $\Gamma$-maximal formulas for $T$ where $\Gamma$ is the set of extensional 1-formulas.
\end{definition}

\begin{proposition}
    The extensional 1-isolators for $\mathbf{S5}$:
\begin{enumerate}
    \item[(i)] $T(x) \wedge F(x)$
    \item[(ii)] $T(x) \wedge \neg F(x)$
    \item[(iii)] $\neg T(x) \wedge F(x)$
    \item[(iv)] $\neg T(x) \wedge \neg F(x)$
\end{enumerate}
\end{proposition}
 
$\mathbf{S5[Con]}$ refutes (i) and makes the additional conjuncts in (ii) and (iii) redundant. Hence, we have the following:

\begin{proposition}
    The extensional 1-isolators for $\mathbf{S5[Con]}$ are:
$$T(x), F(x), N(x)$$
\end{proposition}

There is an analogous notion of intensional 1-isolators:
\begin{definition}
    The \emph{intensional 1-isolators} for $T$ are the $\Gamma$-maximal formulas for $T$ where $\Gamma$ is the set of intensional 1-formulas.
\end{definition}

Let's restrict our attention to $\mathbf{S5[Con]}$ and its extensions. Here is how we characterize the intensional 1-isolators of such a system. We begin by considering all formulas of the form:
$$ ^{+}_{\neg} \lozenge T(x) \wedge ^{+}_{\neg} \lozenge F(x) \wedge ^{+}_{\neg} \lozenge N(x)$$
That is, for each of the extensional 1-isolators, we choose whether it is possible or not and add that choice as a conjunct. Call these the \emph{intensional pre-1-isolators}. To isolate the intensional 1-isolators we must figure out which of the intensional pre-1-isolators are consistent.

$\mathbf{S5[Con]}$ refutes exactly one of the intensional pre-1-isolators.
$$\mathbf{S5[Con]}\vdash \neg \big(\neg \lozenge T(x) \wedge \neg \lozenge F(x) \wedge \neg \lozenge N(x)\big)$$
$\mathbf{S5[Con, Ground]}$ refutes one more in addition:
$$\mathbf{S5[Con, Ground]}\vdash \neg \big(\lozenge T(x) \wedge \lozenge F(x) \wedge \neg \lozenge N(x)\big)$$

\begin{definition}
    A \emph{1-isolator} for $T$ is a $T$-consistent conjunction $\varphi\wedge \psi$ where $\varphi$ is an extensional 1-isolator and $\psi$ is an intensional 1-isolator.
\end{definition}

We are now ready to state Carnap's theorem. To make sense of Carnap's theorem, recall the convention that an empty disjunction of formulas is just a contradiction.

\begin{theorem}[essentially \cite{carnap1946modalities}]\label{carnap}
    For $T$ extending $\mathbf{S5}$:
    \begin{enumerate}
        \item Every extensional 1-formula is $T$ equivalent to a disjunction of extensional 1-isolators for $T$.
        \item Every intensional 1-formula is $T$ equivalent to a disjunction of intensional 1-isolators for $T$.
        \item Every 1-formula is $T$ equivalent to a disjunction of 1-isolators for $T$.
    \end{enumerate}
\end{theorem}

\subsection{Counting}

\cite{carnap1946modalities} calculates the number of formulas up to provable equivalence in the modal \emph{logic} $\mathbf{S5}$. Note the emphasis on the word \emph{logic}. The system that Carnap studies is in the typical language of modal logic, i.e., without the predicates $T,F$ but with propositional variables $p_1,p_2,\dots$. There are far more equivalence classes of 1-formulas in our system $\mathbf{S5}$ than in the modal logic Carnap studied.

\begin{proposition}
        Up to $\mathbf{S5}$-equivalence, there are exactly $4{,}294{,}967{,}296$ many 1-formulas.
\end{proposition}

\begin{proof}
    Let's introduce some abbreviations:
    \begin{flalign*}
        B(x) & \quad T(x) \wedge F(x)\\
        T_{only}(x) & \quad  T(x) \wedge \neg F(x)\\
        F_{only}(x) & \quad  \neg T(x) \wedge  F(x)\\
        N(x) & \quad  \neg T(x) \wedge  \neg F(x)
    \end{flalign*}
    $\mathbf{S5}$ has 4 extensional 1-isolators and 15 intensional 1-isolators, namely, all 1-formulas of the form:
    $$ ^{+}_{\neg} \lozenge B(x) \wedge ^{+}_{\neg} \lozenge T_{only}(x) \wedge ^{+}_{\neg} \lozenge F_{only}(x)\wedge ^{+}_{\neg} \lozenge N(x)$$
except for the one with all negations.

Each extensional 1-isolator is inconsistent with the intensional 1-isolators claiming that it is impossible, and consistent otherwise.

So, for instance, $B(x)$ is consistent with all and only those intensional 1-isolators of the form:
    $$ \lozenge B(x) \wedge ^{+}_{\neg} \lozenge T_{only}(x) \wedge ^{+}_{\neg} \lozenge F_{only}(x)\wedge ^{+}_{\neg} \lozenge N(x)$$
of which there are exactly $2^3=8$.

A 1-isolator is a consistent conjunction of an extensional 1-isolator and an intensional 1-isolator. So there are $4\times 8=32$ many 1-isolators for $\mathbf{S5}$. Hence, the number of 1-formulas up to provable equivalence is exactly $2^{32}=4{,}294{,}967{,}296$.
\end{proof}

\begin{proposition}
        Up to $\mathbf{S5[Con]}$-equivalence, there are exactly $4096$ 1-formulas.
\end{proposition}
\begin{proof}
    We approach this calculation in a similar way. This time the extensional 1-isolators are $T(x),F(x),N(x)$. 
    
    Focus on $T(x)$ for a moment. It is consistent with all and only those intensional 1-isolators of the form:
    $$ \lozenge T(x) \wedge ^{+}_{\neg} \lozenge F(x)\wedge ^{+}_{\neg} \lozenge N(x)$$
of which there are exactly $2^2=4$. 

A 1-isolator is a consistent conjunction of an extensional 1-isolator and an intensional 1-isolator. So there are $3\times 4 = 12$ many 1-isolators for $\mathbf{S5[Con]}$. Hence, the number of formulas up to provable-equivalence is exactly $2^{12}=4096$.
\end{proof}

\begin{proposition}
       Up to $\mathbf{S5[Con, Ground]}$-equivalence, there are exactly $1024$ 1-formulas.
\end{proposition}
\begin{proof}
    Again, the extensional 1-isolators are $T(x),F(x),N(x)$. 
    
    Focus on $T(x)$ for a moment. It is consistent with all and only those intensional 1-isolators of the form:
    $$ \lozenge T(x) \wedge ^{+}_{\neg} \lozenge F(x)\wedge ^{+}_{\neg} \lozenge N(x)$$
    \emph{except} (by the $\mathbf{Ground}$ axiom) for 
       $$ \lozenge T(x) \wedge \lozenge F(x)\wedge \neg \lozenge N(x).$$
So $T(x)$ is consistent with $2^2-1=4-1=3$ intensional 1-isolators. Likewise, $F(x)$ is consistent with $3$ intensional 1-isolators. $N(x)$ is consistent with $4$ (there is no need to delete one since we focus on formulas including $\lozenge N(x)$).

A 1-isolator is a consistent conjunction of an extensional 1-isolator and an intensional 1-isolator. So there are $3+3+4 = 10$ many 1-isolators for $\mathbf{S5[Con,Ground]}$. Hence, the number of formulas up to provable-equivalence is exactly $2^{10}=1024$.
\end{proof}

\subsection{Completeness}

In this subsection we present a complete axiomatization of the one-variable fragment of the modal theory of $\mathbf{FIX}$. Completeness follows easily from Carnap's normal form theorem. The connection between normal forms and completeness theorems was explored previously in \cite{fine1975normal}. We extend this result to the full modal language in a later section.

\begin{lemma}\label{isolator-completeness}
    For any 1-isolator $\varphi$, if $\varphi$ is $\mathbf{S5[Con, Ground]}$-consistent, then there is some $\star$ and some $w$ such that $\mathbf{FIX},w\Vdash_{V^\star} \varphi$.
\end{lemma}

\begin{proof}
    Suppose $\varphi$ is a 1-isolator, i.e., $\varphi=\varepsilon(x)\wedge \iota(x)$ where $\varepsilon(x)$ is an extensional 1-isolator and $\iota(x)$ is an intensional 1-isolator. $\mathbf{S5[Con, Ground]}\vdash \varepsilon(x) \to \lozenge \varepsilon(x)$. Hence $\iota(x)$ has $\lozenge \varepsilon(x)$ as a conjunct.

    We now break into cases depending on which extensional 1-isolator $\varepsilon(x)$ is. In each case we will break into subcases, checking each intensional 1-isolator that has $\lozenge \varepsilon(x)$ as a conjunct. In each subcase we find an assignment $\star$ that satisfies $\varphi$.

\emph{Case 1:} Suppose that $\varepsilon(x)$ is $T(x)$. We must check the following intensional 1-isolators:
    \begin{enumerate}
        \item[(i)] $\iota(x)= \lozenge T(x) \wedge \lozenge F(x) \wedge \lozenge N(x)$. Let $x^\star = \tau$. Then there is some $w$ such that $\mathbf{FIX},w\Vdash_{V^\star} \varepsilon(x)\wedge \iota(x)$.
        \item[(ii)] $\iota(x):=\lozenge T(x) \wedge \neg \lozenge F(x) \wedge \lozenge N(x)$. Let $x^\star=\gamma$. Then there is some $w$ such that $\mathbf{FIX},w\Vdash_{V^\star} \varepsilon(x)\wedge \iota(x)$.
        \item[(iii)] $\iota(x):=\lozenge T(x) \wedge \neg \lozenge F(x) \wedge \neg \lozenge N(x)$. Let $x^\star=0=0$. Then for any $w$, $\mathbf{FIX},w\Vdash_{V^\star} \varepsilon(x)\wedge \iota(x)$.
    \end{enumerate}

We leave Cases 2 ($\varepsilon(x)$ is $F(x)$) and 3 ($\varepsilon(x)$ is $N(x)$) to the reader.
\end{proof}

\begin{theorem}
        For all 1-formulas $\varphi\in\mathcal{L}_\Box$, the following are equivalent:
        \begin{enumerate}
            \item[(i)] $\mathbf{FIX}\Vdash \varphi$
            \item[(ii)] $ \mathbf{S5[Con, Ground]}\vdash \varphi$
        \end{enumerate}
\end{theorem}

\begin{proof}
    The soundness direction is Proposition \ref{sound-1}. Let's focus on completeness. Suppose that $\mathbf{S5[Con, Ground]}\nvdash \varphi$. Then, by classical logic, $\mathbf{S5[Con, Ground]}+\neg \varphi$ is consistent. By Theorem \ref{carnap}, 
    $$\mathbf{S5[Con, Ground]}\vdash \neg \varphi \leftrightarrow (\psi_1\vee\dots\vee\psi_k)$$ 
    where each $\psi_i$ is a 1-isolator. By classical logic, at least one of these 1-isolators $\psi_i$ is consistent. By Lemma \ref{isolator-completeness}, we infer that there is some $\star$ such that $\mathbf{FIX}\Vdash_{V^\star} \psi_i$. Hence $\mathbf{FIX}\Vdash_{V^\star} \neg\varphi$, so $\mathbf{FIX}\nVdash \varphi$.
\end{proof}

Before continuing, we note that the normal form theorem also yields an exact characterization of the number of \textbf{FIX}-definable sets.

\begin{proposition}\label{intensional-dfn}
    Every $\mathbf{FIX}$-definable $A\subseteq \mathcal{L}_T$ is definable by a (possibly empty) disjunction of intensional 1-isolators for $\mathbf{S5[Con, Ground]}$.
\end{proposition}
\begin{proof}
    Every $\mathbf{FIX}$-definable $A\subseteq \mathcal{L}_T$ is definable by an intensional 1-formula. Indeed, if $\varphi$ defines $A$, then this means that exactly the elements of $A$ satisfy $\varphi$ at \emph{every} world. So $\Box\varphi$ also defines $A$. The result follows since every intensional 1-formula is equivalent to a disjunction of intensional 1-isolators (see Theorem \ref{carnap}).
\end{proof}

\begin{corollary}
    The number of $\mathbf{FIX}$-definable sets is exactly $64$.
\end{corollary}
\begin{proof}
    There are exactly $6$ intensional 1-isolators for $\mathbf{S5[Con, Ground]}$. $2^6=64$. The result then follows from Proposition \ref{intensional-dfn}.
\end{proof}

\subsection{Descriptive Complexity}

In this subsection we will finally fulfill our promise to prove the following result:
\begin{theorem}
    The set of intrinsic sentences is not $\mathbf{FIX}$-definable.
\end{theorem}
There are multiple ways of verifying this, and we will present two. One could check all 64 single variable definitions, although this would be tedious. Luckily, there are more conceptual arguments.

Every $\mathbf{FIX}$-definable set is definable by an intensional 1-formula. This is because if $\varphi$ $\mathbf{FIX}$-defines a set $\mathcal{A}$, then so does $\Box\varphi$. We have already seen that each intensional 1-formula is $\mathbf{S5[Con, Ground]}$-equivalent to a disjunction of intensional 1-isolators. So, pushing quantifiers in, we see that every modal definition is equivalent to a formula of the following form:
$$\big( {}_{\lozenge}^{\Box} {}_{\neg}^{+} {}_{F}^{T} (x) \wedge \dots \wedge {}_{\lozenge}^{\Box} {}_{\neg}^{+} {}_{F}^{T} (x) \big)\vee \dots \vee ({}_{\lozenge}^{\Box} {}_{\neg}^{+} {}_{F}^{T} (x) \wedge \dots \wedge {}_{\lozenge}^{\Box} {}_{\neg}^{+} {}_{F}^{T} (x)\big)$$
Let's say that any such formula is in \emph{basic form}.

\begin{definition}
Let $\mathsf{Fix}$ be an $\mathcal{L}_2$ formula defining the set of fixed points. We define a translation that maps $\star:\mathcal{L}_\Box$ formulas in basic form to $\mathcal{L}_2$ formulas.

$(A \vee \dots \vee A)^\star := A^\star \vee \dots \vee A^\star$

$(\Box _{F}^{T} (x))^\star := \forall X\big( \mathsf{Fix}(X) \to \ulcorner _{F}^{T} (x)^\star \urcorner \in X\big)$

$(\Box \neg _{F}^{T} (x))^\star := \forall X\big( \mathsf{Fix}(X) \to \ulcorner _{F}^{T} (x)^\star \urcorner \notin X\big)$

$(\lozenge _{F}^{T} (x))^\star := \exists X\big( \mathsf{Fix}(X) \wedge \ulcorner _{F}^{T} (x)^\star \urcorner \in X\big)$

$(\lozenge \neg _{F}^{T} (x))^\star := \exists X\big( \mathsf{Fix}(X) \wedge \ulcorner _{F}^{T} (x)^\star \urcorner \notin X\big)$

$T(x)^\star = x$

$F(x)^\star=Neg(x^\star)$
\end{definition}

\begin{lemma}
    If $\varphi$ modally defines $A$, then $\varphi^\star$ defines $A$ over the standard interpretation of second-order arithmetic.
\end{lemma}

\begin{proof}
    This follows directly by analyzing the semantics of the modal language.
\end{proof}

\begin{theorem}
    Every modally definable set is $\mathsf{Boolean}\text{-}\Sigma^1_1$-definable in the language of second-order arithmetic.
\end{theorem}

\begin{proof}
    The set of fixed points is $\Delta^1_1$-definable (see \cite[\textsection 6]{burgess1986truth}). In particular, we can choose a $\Sigma^1_1$ formula $\mathsf{Fix}$ defining the set of fixed points.
\end{proof}

The desired indefinability result then follows from the following theorem:

\begin{theorem}[\cite{burgess1986truth}]
    The set of intrinsic sentences is $\Sigma^1_1$-complete in a $\Pi^1_1$ parameter.
\end{theorem}

\begin{corollary}
    The set of intrinsic sentences is not modally definable.
\end{corollary}

\begin{proof}
    No $\Sigma^1_1$-complete in a $\Pi^1_1$ parameter set admits a $\mathsf{Boolean}\text{-}\Sigma^1_1$ definition.
\end{proof}

\subsection{Indiscernibility}

The previous result demonstrates the non-definability of a great number of sets. Nevertheless, it is a coarse tool. We can prove the indefinability of the intrinsic sentences in a more informative way. 

Informally, sets $\mathcal{A}$ and $\mathcal{B}$ are \textbf{$\mathbf{FIX}$-indiscernible} just in case every formula that applies to the $\mathcal{A}$s applies to the $\mathcal{B}$s and vice-versa. We will prove that the intrinsic sentences and the inevitable sentences are $\mathbf{FIX}$-indiscernible, whence the set of intrinsic sentences---the smaller of the two sets---cannot be definable.

\begin{definition}\label{indiscernibility}
    $\mathcal{A}$ and $\mathcal{B}$ are $\mathbf{FIX}$\emph{-indiscernible} if for every $\varphi(x)\in \mathcal{L}_\Box$, TFAE:
    \begin{enumerate}
        \item For all $\star$ with $x^\star\in \mathcal{A}$, $\mathbf{FIX}\Vdash_\mathbf{V^\star} \varphi(x)$.
        \item For all $\star$ with $x^\star \in \mathcal{B}$, $\mathbf{FIX}\Vdash_\mathbf{V^\star} \varphi(x)$.
    \end{enumerate}
\end{definition}

\begin{theorem}\label{ind}
    The set of intrinsic sentences and the set of inevitable sentences are $\mathbf{FIX}$-indiscernible.
\end{theorem}

\begin{proof}
    Let $\mathcal{A}$ be the set of intrinsic and $\mathcal{B}$ the set of inevitable sentences. Then (2) implies (1) (Definition \ref{indiscernibility}) is trivial.

    It suffices to check that (1) implies (2). So assume (1), i.e., 
    $$\oplus \quad \text{For all $\star$ with $x^\star$ intrinsic, $\mathbf{FIX}\Vdash_\mathbf{V^\star} \varphi(x)$}.$$ 
    We may assume that $\varphi$ is an intensional formula (if it is not, just stick a $\Box$ in front of it). We may then assume that $\varphi$ is in normal form, i.e., $\varphi$ is a disjunction of intensional 1-isolators.

    Now the intensional 1-isolators are the following six formulas:
    \begin{enumerate}
        \item[(i)] $\psi_1(x):=\lozenge T(x) \wedge \lozenge F(x) \wedge \lozenge N(x)$
        \item[(ii)] $\psi_{2}(x):=\lozenge T(x) \wedge \neg \lozenge F(x) \wedge \lozenge N(x)$
        \item[(iii)] $\psi_{3}(x):=\neg \lozenge T(x) \wedge \lozenge F(x) \wedge \lozenge N(x)$
        \item[(iv)] $ \psi_{4}(x):=\lozenge T(x) \wedge \neg \lozenge F(x) \wedge \neg \lozenge N(x)$
        \item[(v)] $ \psi_{5}(x):=\neg \lozenge T(x) \wedge  \lozenge F(x) \wedge \neg \lozenge N(x)$
        \item[(vi)] $ \psi_{6}(x):=\neg \lozenge T(x) \wedge  \neg \lozenge F(x) \wedge  \lozenge N(x)$
    \end{enumerate}

Note that the set of inevitable sentences is defined by the disjunction $\psi_2(x)\vee \psi_3(x) \vee \psi_4(x)\vee \psi_5(x)$.

Suppose, for the sake of contradiction, that (2) fails. Then there is some $\star$ such that $y^\star=\psi$ is inevitable and that $\mathbf{FIX}\nVdash_\mathbf{V^\star} \varphi(y)$. So $\varphi$ must be missing one of the disjuncts required to define the inevitable sentences, i.e., $\psi_2(x)$,\dots,$\psi_5(x)$. Let's break into cases.

\emph{Case 1:} $\varphi$ is missing disjunct $\psi_2$. But $\gamma$ is intrinsic and $\psi_2(\gamma)$. So if $x^\star$ is $\gamma$, then $\mathbf{FIX}\nVdash_\mathbf{V^\star} \varphi(x)$, contra $\oplus$.

\emph{Case 2:} $\varphi$ is missing disjunct $\psi_3$. But $\neg \gamma$ is intrinsic and $\psi_3(\neg \gamma)$. So if $x^\star$ is $\neg \gamma$, then $\mathbf{FIX}\nVdash_\mathbf{V^\star} \varphi(x)$, contra $\oplus$.

\emph{Case 3:} $\varphi$ is missing disjunct $\psi_4$. But $0=0$ is intrinsic and $\psi_4(0=0)$. So if $x^\star$ is $0=0$, then $\mathbf{FIX}\nVdash_\mathbf{V^\star} \varphi(x)$, contra $\oplus$.

\emph{Case 4:} $\varphi$ is missing disjunct $\psi_5$. But $0=1$ is intrinsic and $\psi_5(0=1)$. So if $x^\star$ is $0=1$, then $\mathbf{FIX}\nVdash_\mathbf{V^\star} \varphi(x)$, contra $\oplus$.
\end{proof}

\begin{remark}
    If $\mathcal{A}$ and $\mathcal{B}$ are $\mathbf{FIX}$-indiscernible and $\mathcal{A}\subset \mathcal{B}$, then $\mathcal{A}$ is not $\mathbf{FIX}$-definable. Hence Theorem \ref{ind} entails that the set of intrinsic sentences is not $\mathbf{FIX}$-definable.
\end{remark}



\section{Multivariable Definability}\label{mvd}

The logic and combinatorics become more involved---but also more interesting---when there are many variables. 1-formulas are used to define \emph{sets} of sentences. 2-formulas (i.e., formulas with two distinct variables) are used to define \emph{binary relations} on sentences. Likewise, $n$-formulas (i.e., formulas with $n$ distinct variables) are used to define \emph{$n$-ary relations} on sentences.

\begin{definition}
An $\mathcal{L}_\Box$ formula $\varphi(x_1,\dots,x_n)$ \emph{defines} $\mathcal{A}$ over $\mathbf{FIX}$ just in case for every $\mathcal{L}_T$-realization $\star$:
$$\mathbf{FIX}\Vdash_{V^\star} \varphi(x_1,\dots,x_n) \Longleftrightarrow \langle x_1^\star,\dots,x_n^\star\rangle \in \mathcal{A}.$$
\end{definition}

\subsection{Counting Formulas}
Let's begin with some relatively easy counting problems. We will turn to more difficult counting problems in a later subsection.

\subsubsection{The Bare Modal System}

Carnap's normal form theorems generalize to $n$ variables. 
 We summarize the important information for $\mathbf{S5}$ with the following bullets:
\begin{itemize}
    \item The \emph{extensional $n$-isolators} are the formulas of the form:
    $$\mathrel{\substack{B \\ T_{only} \\ F_{only}  \\ N}}  (x_1) \wedge \dots \wedge \mathrel{\substack{B \\ T_{only} \\ F_{only}  \\ N}} (x_n)$$
    \item The \emph{intensional pre-n-isolators} are the formulas of the form:
    $$^{+}_{\neg} \lozenge \varphi_1 \wedge \dots \wedge ^{+}_{\neg} \lozenge \varphi_{4^n}$$ 
    where $\varphi_1,\dots,\varphi_{4^n}$ are the extensional $n$-isolators.
    \item $\mathbf{S5}$ refutes exactly one of these, namely:
    $$\neg \lozenge \varphi_1 \wedge \dots \wedge \neg\lozenge \varphi_{4^n}$$
    and the rest are the intensional $n$-isolators.
    \item The $n$-isolators are the consistent conjunctions $\varphi\wedge\psi$ for $\varphi$ an extensional $n$-isolator and $\psi$ an intensional $n$-isolator.
\end{itemize}

With this terminology on board, we can state the general version of Carnap's theorem:
\begin{theorem}[essentially \cite{carnap1946modalities}]
    For $T$ extending $\mathbf{S5}$:
    \begin{enumerate}
        \item Every extensional $n$-formula is $T$ equivalent to a disjunction of extensional $n$-isolators for $T$.
        \item Every intensional $n$-formula is $T$ equivalent to a disjunction of intensional $n$-isolators for $T$.
        \item Every $n$-formula is $T$ equivalent to a disjunction of $n$-isolators for $T$.
    \end{enumerate}
\end{theorem}

Let's solve a counting problem.

\begin{proposition}
    Up to $\mathbf{S5}$-provable equivalence, the number of $n$-formulas is exactly: $2^{(4^n 2^{(4^n -1)})}$.

\end{proposition}

\begin{proof}
$\mathbf{S5}$ has $4^n$ extensional $n$-isolators and $2^{(4^n)}-1$ intensional $n$-isolators.

Each extensional $n$-isolator is inconsistent with the intensional 1-isolators claiming that it is impossible, and consistent otherwise.

So fixing a particular extensional $n$-isolator $\varepsilon$, there are $2^{(4^n -1)}$ many intensional $n$-isolators with $\lozenge \varepsilon$ as a conjunct.

So there are $4^n \times 2^{(4^n -1)}$ many $n$-isolators. So the number of $n$-formulas up to provable equivalence is 
$2^{(4^n 2^{(4^n -1)})}$.
\end{proof}

\subsubsection{With the Consistency Axiom}

Let's turn out attention to $\mathbf{S5[Con]}$. We summarize the important information in the following bullets:
\begin{itemize}
    \item The \emph{extensional $n$-isolators} are the formulas of the form:
    $$\mathrel{\substack{T \\ F \\ N}}  (x_1) \wedge \dots \wedge \mathrel{\substack{T \\ F \\ N}} (x_n)$$
    \item The \emph{intensional pre-n-isolators} are the formulas of the form:
    $$^{+}_{\neg} \lozenge \varphi_1 \wedge \dots \wedge ^{+}_{\neg} \lozenge \varphi_{3^n}$$ 
    where $\varphi_1,\dots,\varphi_{3^n}$ are the extensional $n$-isolators.
    \item $\mathbf{S5[Con]}$ refutes exactly one of these, namely:
    $$\neg \lozenge \varphi_1 \wedge \dots \neg\lozenge \varphi_{3^n}$$
    and the rest are the intensional $n$-isolators.
    \item The $n$-isolators are the consistent conjunctions $\varphi\wedge\psi$ for $\varphi$ an extensional $n$-isolator and $\psi$ an intensional $n$-isolator.
\end{itemize}

\begin{proposition}
    Up to $\mathbf{S5[Con]}$-provable equivalence, the number of $n$-formulas is exactly: $2^{(3^n 2^{(3^n -1)})}$.
\end{proposition}

\begin{proof}
$\mathbf{S5[Con]}$ has $3^n$ extensional $n$-isolators and $2^{(3^n)}-1$ intensional $n$-isolators.

Each extensional $n$-isolator is inconsistent with the intensional 1-isolators claiming that it is impossible, and consistent otherwise.

So fixing a particular extensional $n$-isolator $\varepsilon$, there are $2^{(3^n -1)}$ many intensional $n$-isolators with $\lozenge \varepsilon$ as a conjunct.

So there are $3^n \times 2^{(3^n -1)}$ many $n$-isolators. So the number of $n$-formulas up to provable equivalence is 
$2^{(3^n 2^{(3^n -1)})}$.
\end{proof}

\subsection{Matrices}

The combinatorics of definability is much harder when $n>1$. In this section, for the sake of developing intuitions, we focus on the 2-variable case. Let's begin by considering this matrix:
\begin{table}[ht]
    \centering
    \begin{tabular}{ccc}
        TT  & TF & TN   \\
        FT  & FF & FN   \\
        NT  & NF & NN 
    \end{tabular}
    \caption{All extensional 2-isolators.}
    \label{tab:my_label}
\end{table}

We identify each cell of the matrix with an extensional 2-isolator. In particular, we identify the cell AB with the formula $A(x_1)\wedge B(x_2)$. For instance, we identify the cell TN with the formula $T(x_1) \wedge N(x_2)$. This yields an association of extensional 2-isolators with pairs of numbers, via Table \ref{tab:numbering}. We will often consider subsets $\mathcal{S}$ of the $3\times 3$ matrix and we move freely between writing as though $\mathcal{S}$ contains NN, $N(x_1)\wedge N(x_2)$, and (3,3), trusting that no confusion will arise. 
\begin{table}[ht]
    \centering
    \begin{tabular}{ccc}
        1,1  & 1,2 & 1,3   \\
        2,1  & 2,2 & 2,3   \\
        3,1  & 3,2 & 3,3 
    \end{tabular}
    \caption{Standard numbering of the $3\times 3$ matrix.}
    \label{tab:numbering}
\end{table}

Each intensional 2-isolator for $\mathbf{S5[Con]}$ corresponds to a subset of the $3 \times 3$ matrix. In particular, each intensional 2-isolator has the form:
$$\bigwedge \{\lozenge \varphi \mid \varphi\in \mathcal{S}\} \wedge \bigwedge \{\neg\lozenge \varphi \mid \varphi \notin \mathcal{S}\}$$
where $\mathcal{S}$ is some subset of the matrix.

However, it is not the case that \emph{each} subset $\mathcal{S}$ of the matrix yields an intensional 2-isolator for $\mathbf{S5[Con]}$. The issue is that some of the subsets of the matrix correspond to $\mathbf{S5[Con]}$-refutable formulas, and intensional 2-isolators are consistent by definition. For instance, letting $\mathcal{S}$ be the empty set yields the formula:
$$\bigwedge \{\neg\lozenge \varphi \mid \varphi \notin \emptyset\}$$
which is refuted by $\mathbf{S5[Con]}$, whence it is not an intensional 2-isolator. 

The crucial question for us is to identify which subsets of the matrix actually reflects the intensional properties of $\mathcal{L}_T$ sentences. Let's review some examples.

Consider Table \ref{tab:table-true-row}, Table \ref{tab:table-violate}, Table \ref{tab:table-diagonal}, and Table \ref{tab:table-missing-corner}. Each of these images exhibits a subset of the $3\times 3$ matrix. In each case, consider the following question: Is there a variable assignment $\star$ such that $x_1^\star$ and $x_2^\star$ can attain all and only these combinations of values? Put another way, given a subset $\mathcal{S}$ of the $3\times 3$ matrix, is there a variable assignment $\star$ such that $\mathbf{FIX} \Vdash_{V^\star} \bigwedge \{\lozenge \varphi \mid \varphi\in \mathcal{S}\} \wedge \bigwedge \{\neg\lozenge \varphi \mid \varphi \notin \mathcal{S}\}$?

\begin{table}[H]
    \centering
    \begin{tabular}{ccc}
        TT  & TF     & TN            \\
        $\text{\xmark}$  & $\text{\xmark}$  &  $\text{\xmark}$\\
        $\text{\xmark}$ &  $\text{\xmark}$ & $\text{\xmark}$
    \end{tabular}
    \caption{A proper subset of the extensional 2-isolators.}
    \label{tab:table-true-row}
\end{table}

\begin{table}[H]
    \centering
    \begin{tabular}{ccc}
        TT           & TF           & $\text{\xmark}$ \\
        $\text{\xmark}$  & $\text{\xmark}$  &  $\text{\xmark}$\\
        $\text{\xmark}$ &  $\text{\xmark}$ & $\text{\xmark}$
    \end{tabular}
    \caption{A proper subset of the extensional 2-isolators.}
    \label{tab:table-violate}
\end{table}

\begin{table}[H]
    \centering
    \begin{tabular}{ccc}
        TT           &     $\text{\xmark}$       &  $\text{\xmark}$\\
        $\text{\xmark}$  & FF  &  $\text{\xmark}$ \\
        $\text{\xmark}$  &  $\text{\xmark}$ & NN 
    \end{tabular}
    \caption{A proper subset of the extensional 2-isolators.}
    \label{tab:table-diagonal}
\end{table}

\begin{table}[H]
    \centering
    \begin{tabular}{ccc}
        TT           & $\text{\xmark}$ \quad  & TN  \\
       $\text{\xmark}$  &  $\text{\xmark}$ \quad  & FN \\
       NT   &  $\text{\xmark}$ \quad   & $\text{\xmark}$
    \end{tabular}
    \caption{A proper subset of the extensional 2-isolators.}
    \label{tab:table-missing-corner}
\end{table}

The answers are given here:

\textbf{Table \ref{tab:table-true-row}:} Yes. Let $x_1^\star$ be a grounded truth (e.g., $0=0$) and $x_2^\star$ be a truth-teller.

\textbf{Table \ref{tab:table-violate}:} No. Note that the first two columns are non-empty but the third is empty. Thus, the axiom $\mathbf{Ground}$
$$ \big( \lozenge T(x) \wedge \lozenge F(x) \big) \to \lozenge N(x)  $$
rules out this possibility, and $\mathbf{FIX}\Vdash\mathbf{Ground}$.

\textbf{Table \ref{tab:table-diagonal}:} Yes. Let $x_1^\star$ and $x_2^\star$ be the same truth-teller.

\textbf{Table \ref{tab:table-missing-corner}:} No. However, one cannot see that there is no such assignment merely deploying the principles codified in $\mathbf{S5[Con,Ground]}$. Indeed, we must introduce the following axiom:
$$ \mathbf{Min_2} := \big(\lozenge N(x) \wedge \lozenge N(y)\big) \to \lozenge \big(N(x)\wedge N(y)\big)$$
Even though $\mathbf{S5[Con,Ground]}\nvdash \mathbf{Min_2}$, it is still the case that $\mathbf{FIX}\Vdash \mathbf{Min_2}$ (by Remark \ref{min-fix-point}). Note that the axiom $\mathbf{Min_2}$ rules out Table \ref{tab:table-missing-corner} since the third row is non-empty and the third column is non-empty but the corner element NN is missing.

This last example identifies the \emph{only} constraint on intensional 2-isolators that is not codified in the $\mathbf{S5[Con,Ground]}$ axioms.

\subsection{The Prime Conditions}

For the full completeness theorem we need new axioms:
$$\mathbf{Min_n}:= \big(\lozenge N(x_1) \wedge \dots \wedge \lozenge N(x_n)\big) \to \lozenge \big(N(x_1)\wedge\dots\wedge N(x_n)\big)$$

\begin{remark}
    For $n>k$, $\mathbf{Min_n}\vdash \mathbf{Min_k}$. To see this, just pick an instance of $\mathbf{Min_n}$ wherein some variables are repeated. For instance, to see that $\mathbf{Min_3}\vdash \mathbf{Min_2}$ note that the following instance of $\mathbf{Min_3}$:
    $$\big(\lozenge N(x_1) \wedge \lozenge N(x_1) \wedge \lozenge N(x_2)\big) \to \lozenge \big(N(x_1)\wedge N(x_1)\wedge N(x_2)\big)$$
    is (logically equivalent to) an instance of $\mathbf{Min_2}$.
\end{remark}

In the previous subsection we identified intensional $2$-isolators with subsets of the $3\times 3$ matrix. For the $n$-variable case, we must turn from matrices to \emph{tensors}, which are generalizations of matrices beyond two dimensions. In particular, we will be interested in $3^n$ tensors, i.e., $\underbrace{3\times\ldots \times 3}\limits_{\mbox{\scriptsize $n$ times}}$ matrices. 

Let $[3]^n$ be the set of $n$-tuples from the set $\{1,2,3\}$. This is the set of positions in a $3^n$ tensor. We can no longer speak only of rows and columns. Instead, we speak in terms of \emph{slices}. The rows are the 1-slices, the columns are the 2-slices, and so on. How do we speak of the first row, the second row, the third column, etc.? The 1st row is the 1st layer in the 1st slice, the third column is the 3rd layer in the second slice, and so on. We introduce the following notation.

\begin{definition}
    For $k\leq 3$ and $j\leq n$, the \emph{$k^{th}$ layer of the $j^{th}$ slice} (written $j_k$) is 
    $$\{a\in [3]^n \mid \exists b_1,\dots,\exists b_{n-1} \; a=\langle b_1,\dots b_{j-1}, k, b_{j},\dots,b_{n-1} \rangle\}$$
\end{definition}

Note that as $j_k$ is defined, it is a set of \emph{positions} in the $3^n$ tensor. That is, so far we have still only discussed positions. The tensors we are interested in are 0-1 valued tensors, i.e., assignments of 0 and 1 to these positions. It is easiest to think of these as just \emph{subsets} of $[3]^n$, and that is how we will proceed.


We now introduce a definition that will be crucial to the proofs of the remaining theorems:
\begin{definition}
    A tensor $\mathcal{S}\subseteq [3]^n$ meets the \emph{prime conditions} if:
    \begin{enumerate}
        \item $\mathcal{S}$ is non-empty.
        \item For all $j\leq n$, if $\mathcal{S}$ intersects $j_1$ and $j_2$, then $\mathcal{S}$ intersects $j_3$.
        \item For all $t\leq n$, for any $t$ slices $s^1,\dots,s^t$, if $\mathcal{S}$ intersects $s^i_3$ for each $i\leq t$, then $\mathcal{S}$ intersects $s^1_3 \cap \dots \cap s^t_3$.
    \end{enumerate}
\end{definition}

\begin{remark}
    In the case $n=2$, the prime conditions can be stated very simply (note that when we refer to row 1, column 2, etc. we have in mind Table \ref{tab:my_label}). We state them here to aid the reader's intuition. A set $\mathcal{S}$ of extensional 2-isolators meets the \emph{prime conditions} if:
    \begin{enumerate}
        \item $\mathcal{S}$ is non-empty.
        \item 
        \begin{enumerate}
            \item If $\mathcal{S}$ intersects row 1 and row 2, it intersects row 3.
            \item If $\mathcal{S}$ intersects column 1 and column 2, it intersects column 3.
        \end{enumerate}
        \item If $\mathcal{S}$ intersects row 3 and column 3, it contains (3,3).
    \end{enumerate}
\end{remark}

\subsection{The Main Lemma}

The main technical lemma of this paper connects the prime conditions with $\mathbf{FIX}$-satisfiability. That is, this lemma shows that those sets that satisfy the prime conditions are exactly the sets that reflect the intensional properties of $\mathcal{L}_T$ sentences.
\begin{lemma}\label{main-lemma-multi}
        For $\mathcal{S}\subseteq[3]^n$, TFAE:
    \begin{enumerate}
        \item $\mathcal{S}$ meets the prime conditions.
        \item There is an $\mathcal{L}_T$ realization $\star$ such that 
        $$\mathbf{FIX} \Vdash_{V^\star} \bigwedge \{\lozenge \varphi \mid \varphi\in\mathcal{S}\} \wedge \bigwedge \{\neg\lozenge \varphi \mid \varphi \notin \mathcal{S}\}.$$
    \end{enumerate} 
\end{lemma}

\begin{proof} \textbf{(2) implies (1):} This follows from soundness, i.e., 
$$\mathbf{FIX}\Vdash \mathbf{S5[Con,Ground,Min_n]}.$$ See Remark \ref{consistency-remark} and Remark \ref{min-fix-point}.

\textbf{(1) implies (2):} Suppose that $\mathcal{S}$ meets the prime conditions. There are two cases to consider.
\begin{enumerate}
    \item[(I)] $\mathcal{S}$ does not contain $(3,\dots,3)$.
    \item[(II)] $\mathcal{S}$ does contain $(3,\dots,3)$.
\end{enumerate}
\textbf{Case (I):} There are very few ways a set can meet the prime conditions and omit (3,\dots,3). The important thing to note is that:
\begin{claim}\label{containment}
    For some $j\leq n$ and some $i\leq 2$, $\mathcal{S}$ is contained in the $i$th layer of the $j$th slice.
\end{claim}
\begin{proof}
    To see that the claim holds, suppose that for all $j\leq n$, for all $i\leq 2$, $\mathcal{S}$ is not contained in the $i$th layer of the $j$th slice. Since $\mathcal{S}$ meets the prime conditions, $\mathcal{S}$ is non-empty. Thus, for every $j$, $\mathcal{S}$ must intersect \emph{some} layer of the $j$th slice. We reason as follows:
\begin{itemize}
    \item $\mathcal{S}$ is not contained in the first layer of the $j$th slice, so $\mathcal{S}$ intersects either the second or third layer of the $j$th slice.
    \item $\mathcal{S}$ is not contained in the second layer of the $j$th slice, so $\mathcal{S}$ intersects either the first or third layer of the $j$th slice.
\end{itemize}
It follows that either $\mathcal{S}$ intersects $j_3$ or it intersects both $j_1$ and $j_2$. In the latter case, the prime conditions imply that $\mathcal{S}$ intersects $j_3$. So either way, $\mathcal{S}$ intersects  $j_3$.

So, for every $j$, $\mathcal{S}$ intersects $j_3$. The prime conditions imply that $\mathcal{S}$ contains $(3,\dots,3)$, a contradiction. This completes the proof of the claim.
\end{proof}

With the claim on board we now set about finishing \textbf{Case (I)} by induction. The base case is $n=2$.

\textbf{Base Case:} We can split into the subcases provided by the claim ($\mathcal{S}$ is contained either in the first row, the second row, the first column, or the second column). Suppose for simplicity's sake that $\mathcal{S}$ is contained in the first row. There are only so many sets contained in the first row:
\begin{enumerate}
    \item[(i)] $\mathcal{S}$ is the first row. Let $A^\star$ be $0=0$ and $B^\star$ be $\tau$.
    \item[(ii)] $\mathcal{S}$ contains just the leftmost and the rightmost elements of the first row. Let $A^\star$ be $0=0$ and $B^\star$ be $ \gamma$.
    \item[(iii)] $\mathcal{S}$ contains the two rightmost elements of the first row. Let $A^\star$ be $0=0$ and $B^\star$ be $\neg \gamma$.
    \item[(iv)] $\mathcal{S}$ is just the leftmost element. Let $A^\star$ be 0=0 and $B^\star$ be 0=0.
    \item[(v)] $\mathcal{S}$ is just the middle element. Let $A^\star$ be 0=0 and $B^\star$ be 0=1.
    \item[(vi)] $\mathcal{S}$ is just the rightmost element. Let $A^\star$ be 0=0 and $B^\star$ be $\lambda$.
\end{enumerate}
By the prime conditions, $\mathcal{S}$ cannot contain just the two leftmost elements of the first row, so these cases are exhaustive.

The reasoning is similar if $\mathcal{S}$ is contained in the second row, first column, or second column. We leave these cases to the reader.

\textbf{Induction step:} Assume the induction hypothesis: For any $\mathcal{S}\subseteq [3]^n$ meeting the prime conditions, there is an appropriate realization $\star$. Now let $\mathcal{T}\subseteq [3]^{n+1}$ meet the prime conditions. By Claim \ref{containment}, $\mathcal{T}$ is contained in the $i$th layer of the $j$th slice. The key observation is that the $i$th layer of the $j$th slice \emph{is} a $3^{n}$ tensor. More precisely, the $j$th conjunct is the same in all the entries in $j_i$, so ignoring this conjunct in all entries yields a $3^{n}$ tensor $\mathcal{T}\restriction [3]^n$. Apply the induction hypothesis to get an $\mathcal{L}_T$ realization $\star$ such that 
        $$\mathbf{FIX} \Vdash_{V^\star} \bigwedge \{\lozenge \varphi \mid \varphi\in \mathcal{T}\restriction [3]^n\} \wedge \bigwedge \{\neg\lozenge \varphi \mid \varphi \notin \mathcal{T}\restriction [3]^n\}.$$
To extend $\star$ to a realization $\sharp$ for $\mathcal{T}$, just let 
$$x_j^\sharp= \begin{cases} 
      0=0 & \text{ if }i = 1 \\
      0=1 & \text{ if } i =2 \\
      \lambda & \text{ if } i=3
\end{cases}$$

\textbf{Case (II):} Assume that $(3,\dots,3)\in\mathcal{S}$. We must define sentences $A^1,\dots,A^n$. Pick some $k$ such that $2^k>3^n$. Take a set $\{\tau_1,\dots,\tau_k\}$ of independent truth-tellers (see Remark \ref{truth-teller}). There are $2^k$-many formulas of the form:
$$ ^{+}_{\neg} \tau_1 \wedge \dots \wedge ^{+}_{\neg} \tau_k$$
Fix some enumeration $C_1,\dots,C_{2^k}$ of these formulas. We then define the formulas $A^1_1,\dots,A^1_{3^n},\dots,A^n_1,\dots,A^n_{3^n}$ so that $A^i_j$ is a conditional of the form $C_j\to \theta$ where $\theta\in\{0=0,0=1,\lambda\}$.

More precisely, we proceed as follows. We fix a numbering $\Phi$ of the cells in the $3^n$ tensor. Then $\Phi(j)$ is some formula $\varphi_1(x_1)\wedge\dots \wedge\varphi_n(x_n)$. Let $\Phi(j)_i$ be the $i$th conjunct of $\Phi(j)$. We then define: 
    $$A^i_j= C_j \to  \begin{cases} 
      0=0 & \text{ if }\Phi(j)_i = T(x_i) \\
      0=1 & \text{ if }\Phi(j)_i = F(x_i) \\
      \lambda & \text{ if }\Phi(j)_i = N(x_i)
\end{cases}$$
We define $\mathcal{S}^\Phi:=\{k \leq 3^n\mid\Phi(k)\in\mathcal{S}\}$.

We define $A^i$ as follows:
$$A^i\equiv \bigwedge \{A^i_j \mid j \in \mathcal{S}^\Phi\}\wedge \bigwedge \{C_j\to \lambda \mid j\in \{1,\dots,2^k\}-\mathcal{S}^\Phi\}$$

Let $\star$ be a variable assignment such that $x_i^\star=A^i$ for each $i\leq n$.

\begin{claim}\label{positive-claim}
    If $\varphi\in\mathcal{S}$, then $\mathbf{FIX}\Vdash_{V^\star}\lozenge \varphi$.
\end{claim}

\begin{proof}
    Suppose that $\varphi\in\mathcal{S}$. Note that $\varphi$ is a conjunction of the form $\varphi_1(x_1)\wedge\dots\wedge \varphi_n(x_n)$. Note also that $\Phi^{-1}(\varphi)$ is some number $k\leq 3^n$ (by $\Phi^{-1}$ I mean the inverse of $\Phi$; so $\Phi^{-1}(\varphi)$ is just the number that $\Phi$ maps to $\varphi$). Since $C_k$ is a Boolean statement about independent truth-tellers, there is some $w$ such that $\mathbb{N},w\vDash_{\mathsf{K3}} C_k$ (see Remark \ref{truth-teller}).
    
    Let $i\leq n$. Note that $A^i_k$ is a conditional of the form $C_k\to \theta$.
    
    $A^i$ is a conjunction of conditionals with antecedent $C_j$ such that for $j\neq k$, $\mathbb{N},w\Dashv_{\mathsf{K3}} C_j$. Which is to say that for all conjuncts $A^i_j$ of $A^i$ except for $A^i_k$, $\mathbb{N},w\vDash_{\mathsf{K3}} A^i_j$. Thus, $\mathbb{N},w\vDash_{\mathsf{K3}}A^i\Leftrightarrow\mathbb{N},w\vDash_{\mathsf{K3}} A^i_k$. Since $\mathbb{N},w\vDash_{\mathsf{K3}} C_k$, we infer that $\mathbb{N},w\vDash_{\mathsf{K3}}A^i\Leftrightarrow \mathbb{N},w\vDash_{\mathsf{K3}}\theta$.

    We now want to prove that $\mathbf{FIX},w\Vdash_{V^\star}\varphi_i$, where $\varphi_i$ is the $i$th conjunct of $\varphi$. Let's break into cases depending on which formula $\varphi_i$ is.
    
    \emph{Case 1:} Suppose $\varphi_i$ is $T(x_i)$, i.e, $\varphi$ is in the 1st layer of the $i$th slice. Since $k = \Phi^{-1}(\varphi)$, $\Phi(j)_i =T(x_i)$. So, $\theta$ is $0=0$. So $\mathbb{N},w\vDash_{\mathsf{K3}}  A^i$. So $\mathbb{N},w\vDash_{\mathsf{K3}}  \mathsf{True}(\ulcorner A^i\urcorner)$. Since $x_i^\star=A^i$, $\mathbf{FIX},w\Vdash_{V^\star} T(x_i)$.
    
    \emph{Case 2:}  Suppose $\varphi_i$ is $F(x_i)$, i.e, $\varphi$ is in the 2nd layer of the $i$th slice. Since $k = \Phi^{-1}(\varphi)$, $\Phi(j)_i =F(x_i)$. So, $\theta$ is $0=1$.  So $\mathbb{N},w\Dashv_{\mathsf{K3}}  A^i$. So $\mathbb{N},w\vDash_{\mathsf{K3}}  \mathsf{True}(\ulcorner\neg A^i\urcorner)$. Since $x_i^\star=A^i$, $\mathbf{FIX},w\Vdash_{V^\star} F(x_i)$.
    
    \emph{Case 3:}  Suppose $\varphi_i$ is $N(x_i)$, i.e, $\varphi$ is in the 3rd layer of the $i$th slice. Since $k = \Phi^{-1}(\varphi)$, $\Phi(j)_i =N(x_i)$. So, $\theta$ is $\lambda$. So $\mathbb{N},w\uparrow_{\mathsf{K3}}  A^i$. So $\mathbb{N},w\nvDash_{\mathsf{K3}}  \mathsf{True}(\ulcorner A^i\urcorner)$ and $\mathbb{N},w\nvDash_{\mathsf{K3}}  \mathsf{True}(\ulcorner\neg A^i\urcorner)$. Since $x_i^\star=A^i$, $\mathbf{FIX},w\Vdash_{V^\star} N(x_i)$.

    So $\mathbf{FIX},w\Vdash_{V^\star}\varphi$. So for all $v\in\mathbf{FIX}$, $\mathbf{FIX},v\Vdash_{V^\star}\lozenge \varphi$. That is, $\mathbf{FIX}\Vdash_{V^\star}\lozenge \varphi$.
\end{proof}

\begin{claim}\label{negative-claim}
    If $\varphi\notin\mathcal{S}$, then $\mathbf{FIX}\nVdash_{V^\star}\lozenge \varphi$.
\end{claim}

\begin{proof}
We prove the contrapositive. Suppose $\mathbf{FIX}\Vdash_{V^\star} \lozenge \varphi$. Then there is some $w$ such that 
$$(\triangle) \quad \mathbf{FIX},w\Vdash_{V^\star} \varphi.$$ 
We now break into cases depending on whether any of the $C_s$ is \textsc{true} at $w$.


\emph{Case 1:} Suppose $\mathbb{N},w\uparrow_{\mathsf{K3}} C_s$ for all $s\leq 2^k$. We claim that $A^1,\dots,A^n$ all have value \textsc{neither} in $w$.

To see this, note that $A^i$ is a conjunction of conditionals each of which has an antecedent $C_s$, which has value \textsc{neither}. A conditional with a \textsc{neither} antecedent cannot be \textsc{false}, so $A^i$ cannot be \textsc{false}. $A_i$ is \textsc{true} only if all its conjuncts are \textsc{true}. Since $2^k>3^n\geq |\mathcal{S}^\Phi|$, there is some $j\leq 2^k$ such that $j\notin \mathcal{S}^\Phi$. But then by definition of $A^i$, $C_j\to\lambda$ is a conjunct of $A^i$, and this has value \textsc{neither}.

So $\varphi$ is the formula $N(x_1)\wedge \dots\wedge N(x_n)$. Since we are working in \textbf{Case (II)}, we have assumed that this formula belongs to $\mathcal{S}$.

\emph{Case 2:} Suppose $\mathbb{N},w\vDash_{\mathsf{K3}} C_s$ for exactly one $s$. We split into subcases depending on whether $s>3^n$ or $s\leq 3^n$.

\emph{Subcase 2.1:} Suppose $s>3^n$. $A^i$ is a conjunction of conditionals of the form $C_j\to \theta^A_i$ such that for $j\neq s$, $\mathbb{N},w\Dashv_{\mathsf{K3}} C_j$, whence $\mathbb{N},w\vDash_{\mathsf{K3}} A^i_j$. Accordingly, $\mathbb{N},w\uparrow_{\mathsf{K3}}A^i\Leftrightarrow \mathbb{N},w\uparrow_{\mathsf{K3}}\theta$, where $\theta$ is the consequent of $A^i_s$. Since $s>3^n$, $\theta$ is $\lambda$. Hence, $\mathbb{N},w\uparrow_{\mathsf{K3}} A^i$. Hence, $\mathbb{N},w\nVdash_{\mathsf{K3}}  \mathsf{True}(\ulcorner A^i\urcorner)$ and $\mathbb{N},w\nVdash_{\mathsf{K3}}  \mathsf{True}(\ulcorner\neg A^i\urcorner)$. Since $x_i^\star=A^i$, $\mathbf{FIX},w\Vdash_{V^\star} N(x_i)$.

Thus, $\varphi$ is $N(x_1)\wedge\dots\wedge N(x_n)$, i.e., (3, \dots, 3). Once again, since we are working in \textbf{Case (II)}, we have assumed that this formula belongs to $\mathcal{S}$.

\emph{Subcase 2.2:} Suppose that $s\leq 3^n$. As in the previous subcase, $\mathbb{N},w\vDash_{\mathsf{K3}}A^i_s \Leftrightarrow \mathbb{N},w\vDash_{\mathsf{K3}}\theta$ where $\theta$ is the consequent of $ A^i_s$.

Now $A^i$ itself is a conjunction of conditionals which is defined to include $A^i_s$ if and only if $s\in\mathcal{S}^\Phi$. Let's split into subsubcases depending on whether $s\in\mathcal{S}^\Phi$.

\emph{Subsubcase 2.2.1:} Suppose $s\in \mathcal{S}^\Phi$. Then $\mathbb{N},w\vDash_{\mathsf{K3}} A^i\Leftrightarrow \mathbb{N},w\vDash_{\mathsf{K3}}A^i_s$, where $A^i_s$ is $C_s\to \theta$. 

Note that, by the definition of $A^i_s$:
$$ (\circledast) \quad \theta \text{ is } 0=0 \text{ just in case } \Phi(s)_i=T(x_i).$$

Since $\mathbb{N},w\vDash_{\mathsf{K3}}C_s$, we infer that:
$$(\boxminus) \quad \mathbb{N},w\vDash_{\mathsf{K3}} A^i\Leftrightarrow \mathbb{N},w\vDash_{\mathsf{K3}}\theta.$$
We reason as follows:
\begin{flalign*}
    \varphi_i \text{ is }T(x_i) &\Longleftrightarrow \mathbf{FIX},w\Vdash_{V^\star}T(x_i) \text{ by ($\triangle)$}\\
    &\Longleftrightarrow \mathbb{N},w\vDash_{\mathsf{K3}}A^i \text{ since $x_i^\star=A^i$}\\
    &\Longleftrightarrow \mathbb{N},w\vDash_{\mathsf{K3}} \theta \text{ by ($\boxminus$)} \\
    &\Longleftrightarrow \theta \text{ is } 0=0 \text{ since $\theta\in\{0=0,0=1,\lambda \}$}\\
    &\Longleftrightarrow \Phi(s)_i = T(x_i) \text{ by ($\circledast$)}
\end{flalign*}
Exactly analogous arguments deliver exactly analogous conclusions. Summarizing, we have:
\begin{flalign*}
    \varphi_i \text{ is }T(x_i) &\Longleftrightarrow \Phi(s)_i = T(x_i)\\
    \varphi_i \text{ is }F(x_i) &\Longleftrightarrow \Phi(s)_i = F(x_i)\\
    \varphi_i \text{ is }N(x_i) &\Longleftrightarrow \Phi(s)_i = N(x_i)
\end{flalign*}

These equivalences show that $\varphi=\Phi(s)$. It is the assumption of Subsubcase 2.2.1 that $s\in\mathcal{S}^\Phi$, i.e., $\Phi(s)\in\mathcal{S}$, so $\varphi \in\mathcal{S}$, which is desired result.

\emph{Subsubcase 2.2.2:} Suppose $s\notin \mathcal{S}^\Phi$. Then $A^i_s$ is $C_s\to \lambda$. So $\mathbb{N},w\uparrow_{\mathsf{K3}}A^i_s$, whence $\mathbb{N},w\uparrow_{\mathsf{K3}}A^i$. As before, we conclude that $\mathbf{FIX},w\Vdash_{V^\star} N(x_i)$. So $\varphi$ is $N(x_1)\wedge \dots\wedge N(x_n)$. By the assumption of \textbf{Case (II)}, $\varphi\in\mathcal{S}$.

This completes the proof of Claim \ref{negative-claim}.
\end{proof}
Claim \ref{positive-claim} and Claim \ref{negative-claim} jointly yield \textbf{Case (II)} of \textbf{(1) implies (2)}. This completes the proof of Lemma \ref{main-lemma-multi}.
\end{proof}

\subsection{Completeness}
Combining Lemma \ref{main-lemma-multi} (the analysis of $\mathbf{FIX}$-satisfiability in terms of the prime conditions) with Carnap's normal form theorem (Theorem \ref{carnap}), yields a complete axiomatization of the modal theory of $\mathbf{FIX}$.

\begin{theorem}
        For all $n$-formulas $\varphi\in\mathcal{L}_\Box$, the following are equivalent:
        \begin{enumerate}
            \item[(i)] $\mathbf{FIX}\Vdash \varphi$
            \item[(ii)] $ \mathbf{S5[Con, Ground,Min_n]}\vdash \varphi$
        \end{enumerate}
\end{theorem}
\begin{proof}
    The soundness direction is trivial. Let's focus on completeness. Suppose that $ \mathbf{S5[Con, Ground, Min_n]}\nvdash \varphi$. By the normal form theorem, 
$$\mathbf{S5[Con, Ground, Min_n]}\vdash \neg \varphi \leftrightarrow (\psi_1\vee\dots\vee\psi_k)$$ 
where each $\psi_i$ is a disjunction of $n$-isolators. Since $\neg \varphi$ is consistent, at least one of the disjuncts $\psi_i$ is also consistent. It follows that $\psi_i$ meets the prime conditions. By Lemma \ref{main-lemma-multi}, there is an $\mathcal{L}_T$ realization $\star$ such that 
        $$\mathbf{FIX} \Vdash_{V^\star}\psi_i.$$
Hence, we infer that $\mathbf{FIX} \nVdash\varphi.$
\end{proof}



\subsection{Counting Definable Relations}

Lemma \ref{main-lemma-multi} yields an explicit characterization of the $\mathbf{FIX}$-definable binary relations, namely, each $\mathbf{FIX}$-definable binary relation is a disjunction of intensional 2-isolators. The following counting result follows from this characterization. 

\begin{theorem}
    The number of $\mathbf{FIX}$-definable binary relations is exactly $2^{276}$.
\end{theorem}
\begin{proof}
First we need to determine the number of intensional 2-isolators. By Lemma \ref{main-lemma-multi}, this is to count the number of subsets of the $3\times 3$ matrix meeting the prime conditions.

There are $2^{3^2}$ subsets of the $3\times 3$ matrix in total. We must subtract those sets that do not meet the prime conditions:
\begin{enumerate}
    \item The empty set. Subtract 1.
    \item The number of sets that intersect the first two rows but not the third is $(2^3 -1)^2=49$. Indeed, we must pick at least one element from the first row and at least one from the second and none from the third. The first row has three cells. Hence it has $2^3 -1=7$ non-empty subsets. Likewise for the second. So we must subtract $7^2=49$ total.
    \item The number of sets that intersect the first two columns but not the third is likewise $(2^3 -1)^2=49$.
    \item We must add back in all those sets that violated both the rows condition and the columns condition. Note that if a set intersects the first two rows but not the third and the first two columns but not the third then it is a subset of the upper left $2\times 2$ square. Exactly $7$ of these sets intersects the first two rows and first two columns.
    \item The sets that intersect row 3 and column 3 but not (3,3). There are $2^2-1=3$ non-empty subsets of the first two slots of the 3rd row. Multiply this by the $3$ ways of doing the same for the 3rd column. Then multiply this with the number of ways of filling in the upper left $2\times 2$ square, which is $2^4=16$. This makes for $3\times 3\times 16$ total.
\end{enumerate}

 Here is the calculation:
    $$2^9 -1 - 49\times 2 + 7 - 3\times 3\times 16 = 276$$
So there are $276$ intensional 2-isolators. The definable binary relations are exactly the disjunctions of intensional 2-isolators, so there are $2^{276}$ of them. This completes the proof.
\end{proof}

How many $\mathbf{FIX}$-definable $n$-ary relations are there? The following coarse estimate is immediate:
\begin{proposition}
        The number of $\mathbf{FIX}$-definable $n$-ary relations is $2^k$ for some $k$ in the interval $(2^{3^n -1},2^{{3^n}})$.
\end{proposition}
\begin{proof}
    A definable $n$-ary relation is a disjunction of intensional $n$-isolators. So this number is $2^k$ where $k$ is the disjunction of intensional $n$-isolators. So now we will calculate bounds on $k$.
    
    Each intensional $n$-isolator is a subset of the set of extensional $n$-isolators, i.e., subset of the $3^n$ tensor. This yields the upper bound $2^{3^n}$.

    Every subset of the $3^n$ tensor that includes $(3,\dots,3)$ is an intensional $n$-isolator. There are $2^{3^n-1}$ such subsets. This yields the lower bound.
\end{proof}
Clearly this interval is huge (and only grows as $n$ grows). One could perhaps use the inclusion--exclusion principle to count exactly the number of $\mathbf{FIX}$-definable $n$-ary relations, but it is not clear whether the resulting formula would be at all informative. Instead of focusing on the exact number, one could aim to find a more informative estimate, but we will not pursue this further in this paper.


\printbibliography

\end{document}